\documentclass[12pt]{article}

\usepackage{epstopdf}% To incorporate .eps illustrations using PDFLaTeX, etc.
\usepackage{latexsym, amsmath, amscd,amsthm}
\usepackage{amssymb}
\usepackage{latexsym, amsmath, amscd,amsthm,booktabs,array,url}
\usepackage{graphicx,epsf}

\newtheorem{theorem}{Theorem}[section]
\newtheorem{corollary}[theorem]{Corollary}
\newtheorem{proposition}[theorem]{Proposition}
\newtheorem{lemma}[theorem]{Lemma}

\theoremstyle{definition}

\newtheorem{example}[theorem]{Example}
\newtheorem{remark}[theorem]{Remark}

\begin{document}

\title{Laplacian eigenvalue distribution for unicyclic graphs}

\author{Sunyo Moon\footnotemark[1]\,  
and Seungkook Park\footnotemark[2]} %\footnotemark[4]}

\date{}
\renewcommand{\thefootnote}{\fnsymbol{footnote}}
\footnotetext[1]{School of Computational Sciences, Korea Institute for Advanced Study, Seoul, Republic of Korea(symoon@kias.re.kr)}
\footnotetext[2]{Department of Mathematics and Research Institute of Natural Sciences, Sookmyung Women's University, Seoul, Republic of Korea(skpark@sookmyung.ac.kr)}
\renewcommand{\thefootnote}{\arabic{footnote}}

\maketitle

\begin{abstract}
Let $G$ be a unicyclic graph.
In this paper, we provide an upper bound for the number of Laplacian eigenvalues of $G$ within the interval $[0,1)$ in terms of the diameter and the girth of $G$.
\end{abstract}

%%%================================================
\section{Introduction}\label{sec:intro}
%%%================================================
%Graph/
Let $G=(V(G),E(G))$ be a simple graph with vertex set $V(G)$ and edge set $E(G)$.
The {\it degree} of the vertex $v$ is the number of vertices adjacent to $v$.
%Laplacian matrix/
The Laplacian matrix $L(G)$ of $G$ is defined by $L(G)=D(G)-A(G)$, where $D(G)$ is the diagonal matrix of vertex degree and $A(G)$ is the adjacency matrix of $G$.
An eigenvalue of $L(G)$ is called a {\it Laplacian eigenvalue} of $G$.
Since $L(G)$ is symmetric and positive semidefinite, all its eigenvalues are real and non-negative. We shall index the Laplacian eigenvalues of $G$ in non-decreasing order and denote them as 
\begin{equation*}
    0=\mu_1(G) \leq \mu_2(G) \leq \cdots \leq \mu_n(G). 
\end{equation*}
The multiplicity of the Laplacian eigenvalue $\mu$ is denote by $m_G(\mu)$. 
It is well known that the largest Laplacian eigenvalue can not exceed $n$. Let $I$ be an interval in $[0,n]$. We denote the number of Laplacian eigenvalues of $G$ in the interval $I$ by $m_GI$.
For a connected graph $G$, the {\it distance} between two vertices in $G$ is the length of the shortest path containing them.
The largest distance between any two vertices in $G$ is called the {\it diameter} of $G$ and is denoted by $d(G)$. A {\it diametral path} of a graph is a shortest path whose length is equal to the diameter of the graph.
A {\it dominating set} is a subset $D$ of $V(G)$ such that every vertex $u \in V(G)\backslash D$ is adjacent to a vertex in $D$. 
The {\it domination number} of $G$, denoted by $\gamma(G)$, is the minimum cardinality of a dominating set of $G$.
A {\it tree} is a connected graph without any cycles and a {\it unicyclic graph }is a connected graph containing exactly one cycle.
For a graph $G$, a lower bound for the domination number $\gamma(G)$ was given in \cite{Haynes,Kang}, which is $$\frac{d(G)+1}{3} \leq \gamma(G).$$ 
In 2016, Hedetniemi, Jacobs and Trevisan  \cite{Hedetniemi} proved that for any graph $G$, $$m_G[0,1) \leq \gamma(G).$$
Recently, the relation between $\frac{d(T)+1}{3}$ and $m_G[0,1)$ was given by Guo, Xue and Liu in \cite{Guo}. They showed that for any tree $T$, $ \frac{d(T)+1}{3} \leq m_T[0,1)$ and hence 
\begin{equation*}
    \frac{d(T)+1}{3} \leq m_T[0,1) \leq \gamma(T).
\end{equation*}
However, the inequality $\frac{d(G)+1}{3} \leq m_G[0,1)$ does not hold for a general connected graph $G$. For example, the cycle graph $C_6$ with 6 vertices has one eigenvalue in $[0,1)$. However, $\frac{d(C_6)+1}{3}=\frac{4}{3}>1$.
In this paper, we show that if $G$ is a unicyclic graph with girth $r$, then $$\bigg\lceil\frac{d(G)}{3} \bigg\rceil+\bigg\lceil \frac{r}{6} \bigg\rceil-1 \leq m_G[0,1).$$
Hence we obtain the following relation for a unicyclic graph $G$ with girth $r \geq 7$:
\begin{equation*}
    \frac{d(G)+1}{3} \leq \bigg\lceil\frac{d(G)}{3} \bigg\rceil +\bigg\lceil \frac{r}{6} \bigg\rceil-1 \leq m_G[0,1) \leq \gamma(G).
\end{equation*}

%%%================================================
\section{Preliminaries}
%%%================================================
Let $G$ be a graph and let $e$ be an edge of $G$. 
The graph $G-e$ is obtained from $G$ by deleting the edge $e$. 
The following theorem implies that removing an edge from a graph leads to a decrease in the values of the Laplacian eigenvalues of the graph.
\begin{theorem}[\cite{Grone} Theorem 4.1]\label{lem:interlacing}
    Let $G$ be a graph with $n$ vertices and let $e$ be an edge of $G$.
    For $i=1,\dots, n-1$,
    \begin{equation*}
        \mu_i(G) \leq \mu_{i+1}(G-e) \leq \mu_{i+1}(G). 
    \end{equation*}
    
\end{theorem}
The following lemma shows that attaching a pendant vertex to a graph does not change the lower bound for the number of Laplacian eigenvalues in the interval $[0,1)$. 
\begin{lemma}[\cite{Guo} Lemma 2.2]\label{lem:rem-pendant}
    Let $G$ be a graph on $n$ vertices and let $G'$ be a graph obtained from $G$ by deleting a pendant vertex. If $m_{G'}[0,1) \geq k$ then $m_{G}[0,1) \geq k$. 
\end{lemma}

\begin{figure}
    \centering
    \includegraphics[width=16cm]{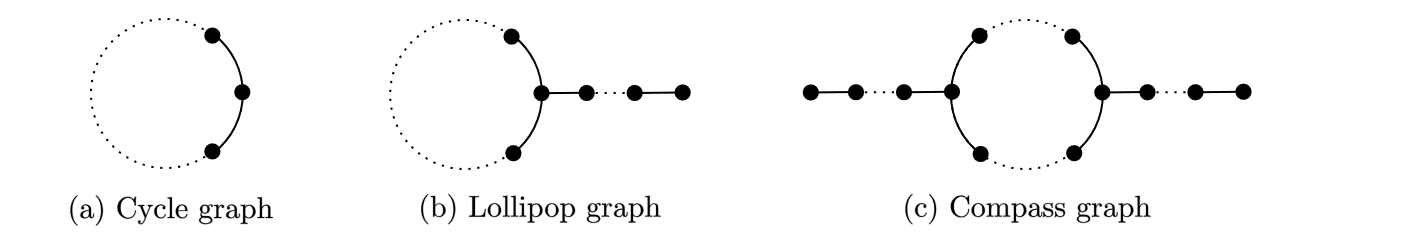}
    \caption{Three types of unicyclic graphs}
    \label{fig:types}
\end{figure}

Let $G$ be a unicyclic graph with $n$ vertices and girth $r \geq 3$.
Then $G$ has $n$ edges and consists of a cycle $C_r$ of length $r~(3 \leq r \leq n)$ and disjoint maximal trees $T_1,\ldots, T_m ~(0 \leq m \leq r)$ such that each tree has exactly one vertex in common with  $C_r$.
Thus, any unicyclic graph $G$ can be obtained by attaching pendant vertices to the minimal unicyclic subgraph of $G$ that contains a diametral path of $G$.
By Lemma~\ref{lem:rem-pendant}, it is enough to consider the minimal unicyclic subgraph of $G$ that contains a diametral path of $G$, which are given in Figure~\ref{fig:types}.
The {\it lollipop graph} is obtained by appending a vertex of the cycle $C_r$ to an end vertex of the path $P_{n-r}$.
The lollipop graph with $n$ vertices and girth $r$ is denoted by $C_{n,r}$.
In order to define the compass graph, we  label the $n$ vertices of the graph as shown in Figure~\ref{fig:compass}. 
Let $t$ and $s$ be positive integers such that $t+s=n-r$. A compass graph consists of a cyclic graph $C_r$ with paths $P_t$ and $P_s$ attached. To be more precise, the {\it compass graph} is obtained by attaching an end vertex of the path $P_t$ to a degree 2 cycle vertex of the lollipop graph $C_{n-t,r}$ such that its distance from the degree 3 vertex is $r'$. We denote this graph by $C_{n,r}(r',t)$. Note that the diameter of $C_{n,r}(r',t)$ is $d(C_{n,r}(r',t))=r'+t+s$. 

\begin{figure}
    \centering
    \includegraphics[width=16cm]{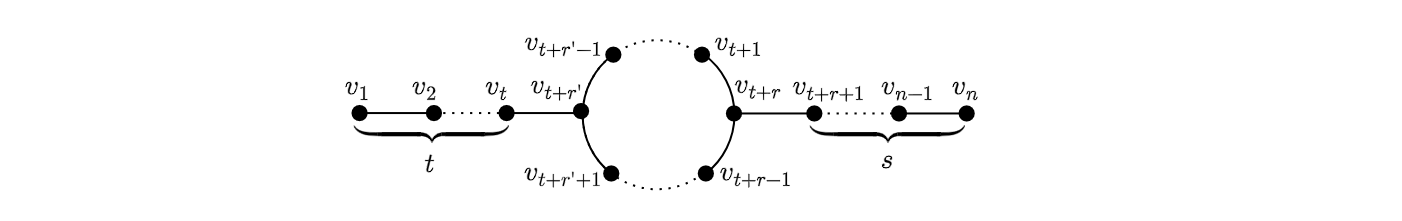}
    \caption{The compass graph $C_{n,r}(r't)$}
    \label{fig:compass}
\end{figure}

Let $G$ be a unicyclic graph and let  $G'$ be the minimal unicyclic subgraph of $G$ that contains a diametral path of $G$.
We will illustrate how to deduce $G'$ from $G$ and conclude that it is sufficient to find the lower bound of $m_{G'}[0,1)$.
\begin{example}
    Let $G$ be a graph with 19 vertices as shown on the left in Figure~\ref{fig:ex1}.
    The diameter of $G$ is 8 and the girth of $G$ is 8.
    The path consisting of the white vertices is the diametral path of $G$.
    Then $G'$ is the lollipop graph $C_{12,8}$.
    Suppose that $m_{C_{12,8}}[0,1)\geq k$. Then, by Lemma~\ref{lem:rem-pendant}, $m_G[0,1) \geq k$.
    Our bound for $m_{C_{12,8}}[0,1)$ is  
    \begin{equation*}
        m_{C_{12,8}}[0,1)\geq \bigg\lceil\frac{8}{3}\bigg\rceil+ \bigg\lceil\frac{8}{6}\bigg\rceil -1=4.
    \end{equation*}
    Hence $m_G[0,1) \geq 4$. In fact, $m_{C_{12,8}}[0,1)=4$ and $m_G[0,1)=6$.
    \begin{figure}[h!t!]
    \centering
    \includegraphics[width=16cm]{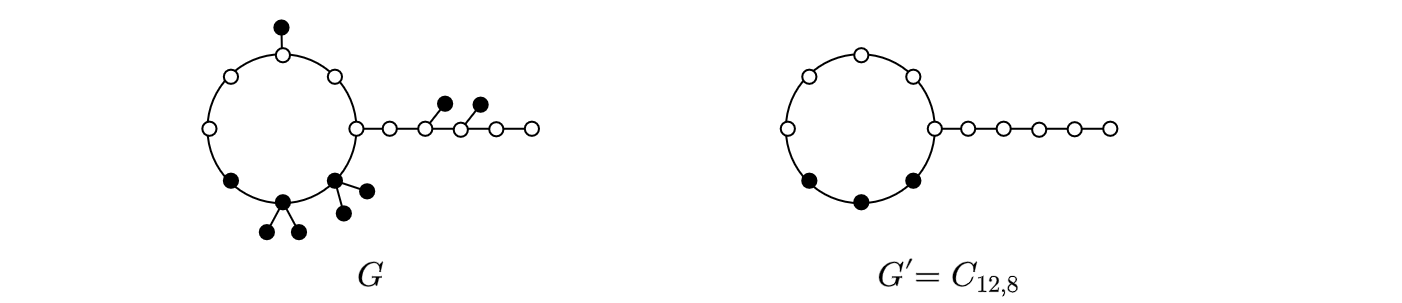}
    \caption{Reducing to lollipop graph}
    \label{fig:ex1}
\end{figure}
\end{example}
\begin{example}
    Let $G$ be a graph with 24 vertices as shown on the left in Figure~\ref{fig:ex2}.
    The diameter of $G$ is 10 and the girth of $G$ is 8. 
    The path consisting of the white vertices is the diametral path of $G$.
    Then $G'$ is the compass graph $C_{14,8}(4,3)$.
    Our bound for $m_{C_{14,8}(4,3)}[0,1)$ is
    \begin{equation*}
        m_{C_{14,8}(4,3)}[0,1)\geq \bigg\lceil\frac{10}{3}\bigg\rceil+ \bigg\lceil\frac{8}{6}\bigg\rceil -1=5. 
    \end{equation*}
    Hence $m_G[0,1)\geq 5$. In fact, $m_{C_{14,8}(4,3)}[0,1)=5$ and $m_G[0,1)=8$.
    \begin{figure}[h!t!]
    \centering
    \includegraphics[width=16cm]{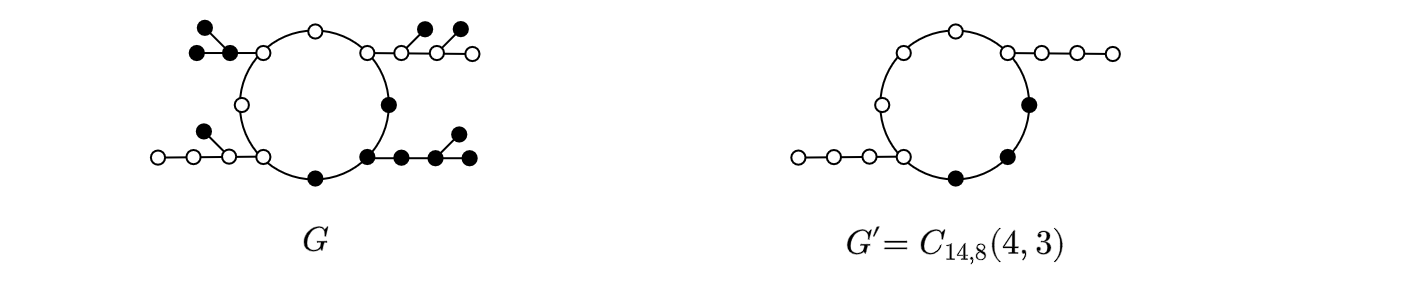}
    \caption{Reducing to compass graph}
    \label{fig:ex2}
\end{figure}
\end{example}

Now, we list some facts and lemmas which are used to prove our main results.
For an $n \times n$ matrix $M$, we denote $\mu_i(M)$ as the $i$th smallest eigenvalue of $M$. 
\begin{lemma}[\cite{Horn} Theorem 4.3.1]\label{lem:eig_sum}
    Let $A$ and $B$ be $n \times n$ Hermitian matrices. 
    For the integers $i$ and $j$ $(1\leq i,j\leq n)$ satisfying $1 \leq i+j-n \leq n$,
    \begin{equation*}
        \mu_{i+j-n}(A+B) \leq \mu_i(A)+\mu_j(B).
    \end{equation*}
    Moreover, equality holds if and only if there exists a unit vector $x$ such that $Ax=\mu_i(A)x$, $Bx=\mu_j(B)x$, and $(A+B)x=\mu_{i+j-n}(A+B)x$.
\end{lemma}

\begin{lemma}[\cite{Grone} Theorem 3.1]
    Let $G$ be a graph with $n$ vertices and suppose that $H$ is a graph obtained from $G$ and $P_3$ by joining a pendant vertex of $P_3$ to a vertex of $G$.
    Then $m_H(1)=m_G(1)$.
\end{lemma}
\begin{corollary}\label{coro:path}
    Let $G$ be a graph with $n$ vertices and let $m$ be a positive integer divisible by $3$. 
    Suppose that $H$ is a graph obtained from $G$ and $P_m$ by joining  a pendant vertex of $P_m$ to a vertex of $G$.
    Then $m_H(1)=m_G(1)$.
\end{corollary}

\begin{remark}
    Since the Laplacian eigenvalues of $P_n$ are
    \begin{equation*}
        2-2\cos{\frac{k\pi}{n}},
    \end{equation*}
    where $k=0,\ldots,n-1$, we have
    \begin{itemize}
        \item[(a)] $m_{P_n}[0,1)=\left\lceil \frac{n}{3} \right\rceil$,
        \item[(b)] $3 \,|\, n$ if and only if $m_{P_n}(1)=1$.
    \end{itemize}
\end{remark}
  
\begin{lemma}\label{lem:path-one}
    Let $P_n$ be the path graph with $n$ vertices and let  $x=\begin{bmatrix} x_1 & \cdots & x_n \end{bmatrix}^T$ be a vector of order $n$, where
    \begin{equation*}
        x_i =\left\{\begin{array}{rl}
                1, & \mbox{if $i\equiv 1,~ 6 \pmod{6}$,}\\
                0, & \mbox{if $i\equiv 2,~ 5 \pmod{6}$,}\\
                -1, & \mbox{if $i\equiv 3,~ 4 \pmod{6}$.}
              \end{array}\right.
    \end{equation*}
    If $n$ is divisible by 3, then the vector $x$
    is an eigenvector corresponding to the Laplacian eigenvalue $1$ of $P_n$.
\end{lemma}
\begin{proof}
    One can immediately verify that $L(P_n)x = x$.
\end{proof}
We denote the characteristic polynomial of a matrix $M$ by $\phi(M)=\phi(M;x)=\det(xI-M)$, where $I$ is the identity matrix. If $M=L(G)$, then we will write $\phi(L(G))$ as $\phi(G)$.
For a vertex $v$ of $G$, let $L_v(G)$ be the principal submatrix of $L(G)$ formed by deleting the row and column corresponding to the vertex $v$.
Let $B_n$ be the matrix of order $n$ obtained from $L(P_{n+1})$ by deleting the row and column corresponding to one of the end vertices of $P_{n+1}$. Let $H_n$ be the matrix of order $n$ obtained from $L(P_{n+2})$ by deleting the rows and columns corresponding to both end vertices of $P_{n+2}$.
\begin{lemma}[\cite{GuoJM05} Lemma 8]\label{lem:char-joining}
    Let $G$ be the graph obtained by joining the vertex $u$ of the graph $G_1$ to the vertex $v$ of the graph $G_2$ by an edge. Then
    \begin{equation*}
        \phi(G)=\phi(G_1)\phi(G_2)-\phi(G_1)\phi(L_v(G_2))-\phi(G_2)\phi(L_u(G_1)).
    \end{equation*}
\end{lemma}

\begin{lemma}[\cite{GuoJM08} Lemma 2.8]\label{lem:char-PC}
Set $\phi(P_0)=0$, $\phi(B_0)=1$, $\phi(H_0)=1$. We have
\begin{itemize}
    \item[(a)] $x\phi(B_n)=\phi(P_{n+1})+\phi(P_n)$;
    %\item[(b)] $\phi(P_{n+1})=(x-2)\phi(P_n)-\phi(P_{n-1})$, $(n\geq 1)$;
    \item[(b)] $\phi(P_n)=x\phi(H_{n-1})$, $(n\geq 1)$;
    \item[(c)] $\phi(C_n)=\frac{1}{x}\phi(P_{n+1})-\frac{1}{x}\phi(P_{n-1})+2(-1)^{n+1}$, $(n \geq 3, x\neq 0)$.
\end{itemize}
\end{lemma}

\begin{lemma}\label{lem:char-lollipop}
    Let $C_{n,r}$ be a lollipop graph with $n$ vertices and girth $r$. Then the characteristic polynomial of $C_{n,r}$ can be represented as
    \begin{equation*}
        \phi(C_{n,r})=\phi(P_{n-r})\bigg( \phi(C_{r})-\frac{1}{x}\phi(C_{r}) \bigg)-\frac{1}{x}\phi(C_{r})\phi(P_{n-r-1})-\frac{1}{x}\phi(P_{n-r})\phi(P_{r}).
    \end{equation*}
\end{lemma}
\begin{proof}
    Let $u$ be a vertex of $C_r$ and let $v$ be an end vertex of $P_{n-r}$.
    Then, by Lemma \ref{lem:char-joining},
    \begin{equation*}
        \phi(C_{n,r})=\phi(C_{r})\phi(P_{n-r})-\phi(C_{r})\phi(B_{n-r-1})-\phi(P_{n-r})\phi(H_{r-1}).
    \end{equation*}
    By Lemma \ref{lem:char-PC} (a) and (b), we have
    \begin{align*}
    \phi(C_{n,r})=\,&\phi(C_{r})\phi(P_{n-r})-\frac{1}{x}\phi(C_{r})\big( \phi(P_{n-r})+\phi(P_{n-r-1})\big)-\frac{1}{x}\phi(P_{n-r})\phi(P_r)\\
    =\,&\phi(P_{n-r})\bigg( \phi(C_{r})-\frac{1}{x}\phi(C_{r}) \bigg)-\frac{1}{x}\phi(C_{r})\phi(P_{n-r-1})-\frac{1}{x}\phi(P_{n-r})\phi(P_{r}).
    \end{align*}
\end{proof}

%%%================================================
\section{Main results}\label{sec:main}
%%%================================================
In this section, we establish a lower bound for $m_G[0,1)$, where $G$ represents the cycle, lollipop, and compass graphs. For the cycle graph, the exact value of $m_G[0,1)$ can be computed, as shown in the following remark.
\begin{remark}\label{rem:cycle}
    Since the Laplacian eigenvalues of $C_n$ are
    \begin{equation*}
        2-2\cos{\frac{2 \pi k}{n}},%=4\sin^2{\frac{k\pi}{n}},
    \end{equation*}
    where $k=0,\ldots,n-1$, we have 
    \begin{itemize}
        \item[(a)] $m_{C_n}[0,1)=2\left\lceil \frac{n}{6}\right\rceil-1$,
        \item[(b)]  $6\,|\,n$ if and only if $m_{C_n}(1)=2$.
    \end{itemize}
\end{remark}

    By Remark~\ref{rem:cycle}, we have $m_{C_n}[0,1) =2\lceil \frac{n}{6}\rceil-1 $.
    Since the diameter of $C_n$ is $\lfloor \frac{n}{2} \rfloor$, we obtain
    \begin{align*}
        &2\bigg\lceil \frac{n}{6}\bigg\rceil-1 -\bigg(\bigg\lceil\frac{1}{3}\left\lfloor \frac{n}{2} \right\rfloor\bigg\rceil+\bigg\lceil\frac{n}{6}\bigg\rceil-1\bigg)
        =\bigg\lceil \frac{n}{6}\bigg\rceil-\bigg\lceil \frac{1}{3}\left\lfloor \frac{n}{2} \right\rfloor\bigg\rceil
        \geq\bigg\lceil \frac{n}{6}\bigg\rceil-\bigg\lceil\frac{1}{3}\left\lceil \frac{n}{2} \right\rceil\bigg\rceil= 0.
    \end{align*}
Thus 
\begin{equation*}
        m_{C_n}[0,1) \geq \bigg\lceil\frac{d(C_n)}{3}\bigg\rceil+ \bigg\lceil\frac{n}{6}\bigg\rceil-1.
    \end{equation*}

\subsection{Lollipop graphs}
\begin{figure}
    \centering
    \includegraphics[width=16cm]{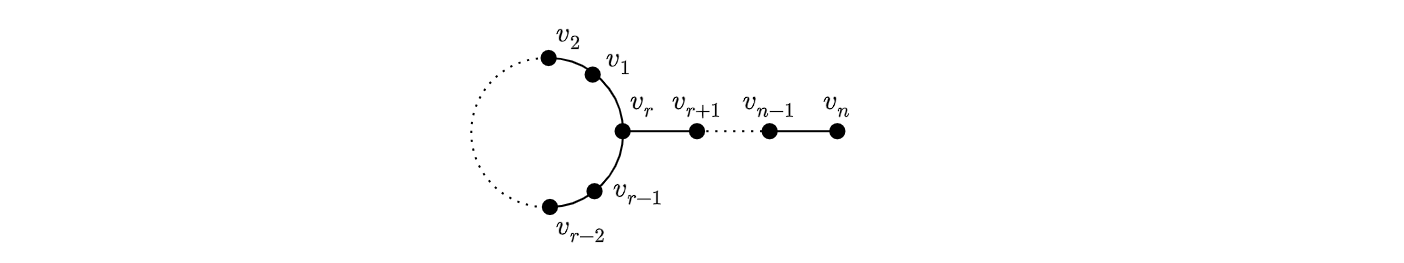}
    \caption{Lollipop graph $C_{n,r}$}
    \label{fig:lollipop}
\end{figure}

We consider the lollipop graph $C_{n,r}$ with $n$ vertices and girth $r$. We label the vertices as shown in  Figure~\ref{fig:lollipop}. Then the diameter of $C_{n,r}$ is $d(C_{n,r})=\lfloor \frac{r}{2} \rfloor +(n-r)=n-\lceil \frac{r}{2} \rceil$.
Note that  $m_{G_1\cup G_2}(\mu)=m_{G_1}(\mu)+m_{G_2}(\mu)$ for any disjoint graphs $G_1$ and $G_2$, where $\mu$ is a Laplacian eigenvalue of $G_1 \cup G_2$.
\begin{proposition}\label{prop:lollipop}
    Let $C_{n,r}$ be the lollipop graph with $n$ vertices and girth $r$. Then
\begin{equation*}
    m_{C_{n,r}}[0,1) \geq \bigg\lceil\frac{d(C_{n,r})}{3} \bigg\rceil +\bigg\lceil \frac{r}{6} \bigg\rceil -1.
\end{equation*}
\end{proposition}
\begin{proof}
    Let $k=\big\lceil \frac{d(C_{n,r})}{3} \big\rceil +\left\lceil \frac{r}{6} \right\rceil -1$. 
    We consider the following two subgraphs:
    $$P_{n}=C_{n,r}-v_1v_r ~~\text{and}~~ C_r \cup P_{n-r}=C_{n,r}-v_rv_{r+1},$$
    where $v_1,v_r,v_{r+1}$ are 
    labeled as in Figure~\ref{fig:lollipop}.
    
    {\bf Case 1.} Suppose that $r \equiv 1 \pmod{6}$. Then $\lceil \frac{r-1}{2} \rceil=\lfloor \frac{r}{2} \rfloor$. 
    We have
    \begin{align*}
        m_{C_{r}\cup P_{n-r}}[0,1)& =m_{C_{r}}[0,1)+m_{P_{n-r}}[0,1)\\
        &=2\bigg\lceil\frac{r}{6}\bigg\rceil -1 +\left\lceil \frac{n-r}{3} \right\rceil\\
        &=\left\lceil \frac{n-r}{3} \right\rceil + \left\lceil \frac{r-1}{6} \right\rceil+\bigg\lceil \frac{r}{6} \bigg\rceil\\
        &\geq \left\lceil \frac{n-r+\left\lceil \frac{r-1}{2} \right\rceil}{3} \right\rceil+\bigg\lceil \frac{r}{6} \bigg\rceil\\
        &=\left\lceil \frac{n-r+\lfloor \frac{r}{2} \rfloor}{3} \right\rceil+\bigg\lceil \frac{r}{6} \bigg\rceil\\
        &=\left\lceil \frac{d(C_{n,r})}{3} \right\rceil+\bigg\lceil \frac{r}{6} \bigg\rceil=k+1.
    \end{align*}
    By Lemma \ref{lem:interlacing}, we have $\mu_k(C_{n,r})\leq \mu_{k+1}(C_r \cup P_{n-r}) <1$. Thus $m_{C_{n,r}}[0,1) \geq k$.
    
    {\bf Case 2.} Suppose that $r\equiv 2\pmod{6}$. Then $\lceil \frac{r}{2} \rceil=\lfloor\frac{r}{2}\rfloor$.
    Note that 
    \begin{align*}
        m_{C_r \cup P_{n-r}}[0,1)=& \,2\bigg\lceil \frac{r}{6} \bigg\rceil-1 +\left\lceil \frac{n-r}{3} \right\rceil\\
        =&\left\lceil \frac{n-r}{3} \right\rceil+\bigg\lceil \frac{r}{6} \bigg\rceil+\bigg\lceil \frac{r}{6} \bigg\rceil-1 \\
        \geq &\left\lceil \frac{n-r+\lceil \frac{r}{2} \rceil}{3} \right\rceil+\bigg\lceil \frac{r}{6} \bigg\rceil-1\\
        =&\left\lceil \frac{n-r+\lfloor \frac{r}{2} \rfloor}{3} \right\rceil+\bigg\lceil \frac{r}{6} \bigg\rceil-1\\
        =&\left\lceil \frac{d(C_{n,r})}{3} \right\rceil+\bigg\lceil \frac{r}{6} \bigg\rceil-1=k.
    \end{align*}
    If $n-r \not\equiv 0 \pmod{3}$ then $m_{C_r \cup P_{n-r}}[0,1) >k$ since $\left\lceil \frac{n-r}{3} \right\rceil+\left\lceil \frac{r}{6} \right\rceil > \left\lceil \frac{n-r+\lfloor \frac{r}{2} \rfloor}{3} \right\rceil$. 
    Thus $\mu_k(C_{n,r})\leq \mu_{k+1}(C_r \cup P_{n-r}) <1$ and hence $m_{C_{n,r}}[0,1) \geq k$.
    Now, we consider the case $n-r \equiv 0 \pmod{3}$.
    Since $\left\lceil \frac{n-r}{3} \right\rceil+\left\lceil \frac{r}{6} \right\rceil = \left\lceil \frac{n-r+\lfloor \frac{r}{2} \rfloor}{3} \right\rceil$, it follows that $m_{C_r \cup P_{n-r}}[0,1) =k$ . 
    Suppose that $m_{C_{n,r}}[0,1)<k$. Then $\mu_k(C_{n,r}) \geq 1$. 
    Since $n-r \equiv 0 \pmod{3}$, $P_{n-r}$ has a Laplacian eigenvalue 1. Hence $\mu_{k+1}(C_r \cup P_{n-r})=1$.
    It follows that $\mu_k(C_{n,r})=\mu_{k+1}(C_r \cup P_{n-r})=1$ by Lemma \ref{lem:interlacing}.
    Note that $L(C_{n,r})-L(C_r \cup P_{n-r})=M=(m_{ij})$, where
    \begin{equation*}
        m_{ij}=\left\{\begin{array}{rl}
            1, & \mbox{if $(i,j) \in \{(r,r),~(r+1,r+1)\}$,}\\
            -1, & \mbox{if $(i,j) \in \{(r,r+1),~(r+1,r)\}$,}\\
            0, & \mbox{otherwise.}
        \end{array}\right.
    \end{equation*}
    Since $M$ is permutationally similar to $L(K_2 \cup (n-2)K_1)$, we have 
    \begin{equation*}
        \mu_{1}(M)=\cdots=\mu_{n-1}(M)=0~~\text{and}~~\mu_n(M)=2.
    \end{equation*}
    Then $\mu_k(C_{n,r})=\mu_{k+1}(C_r \cup P_{n-r})+\mu_{n-1}(M)$.
    By Lemma~\ref{lem:eig_sum}, there exists an eigenvector $z$ such that $$L(C_{n,r})z=\mu_k(C_{n,r})z=z,$$ $$L(C_r \cup P_{n-r})z=\mu_{k+1}(C_r \cup P_{n-r})z=z,~~\text{and}$$ $$Mz=\mu_{n-1}(M)z=0.$$
    Let $x$ be the eigenvector corresponding to the Laplacian eigenvalue 1 of $P_{n-r}$, which is given in Lemma~\ref{lem:path-one}.
    Since 1 is not a Laplacian eigenvalue of $C_r$, the eigenvector $z$ corresponding to the Laplacian eigenvalue 1 of $C_r \cup P_{n-r}$ is a multiple of
    \begin{equation*}
        \left[\begin{array}{c}
             0  \\\hline
             x 
        \end{array}\right].
    \end{equation*}
    Since the $r$th coordinate of $z$ is 0 and the $(r+1)$th coordinate of $z$ is nonzero, we have $Mz \neq 0$. This is a contradiction and hence $m_{C_{n,r}}[0,1) \geq k$.
    
    {\bf Case 3.} Suppose that $r \equiv 3,4 \pmod{6}$.
    Then $\lceil \frac{r}{2} \rceil \equiv 2 \pmod{3}$.
    If $n \not\equiv 0 \pmod{3}$, then 
    \begin{align*}
        m_{P_n}[0,1)=&\,\bigg\lceil \frac{n}{3} \bigg\rceil =\left\lceil \frac{n+1}{3} \right\rceil\\
        =& \left\lceil \frac{d(C_{n,r})+\lceil \frac{r}{2} \rceil+1}{3} \right\rceil\\
        =& \left\lceil \frac{d(C_{n,r})}{3} \right\rceil+ \left\lceil \frac{\lceil \frac{r}{2} \rceil+1}{3} \right\rceil\\
        =&\left\lceil \frac{d(C_{n,r})}{3} \right\rceil+ \left\lceil \frac{r+2}{6} \right\rceil\\
        =&\left\lceil \frac{d(C_{n,r})}{3} \right\rceil+ \bigg\lceil \frac{r}{6} \bigg\rceil =k+1.
    \end{align*}
    Thus $\mu_k(C_{n,r})\leq \mu_{k+1}(P_{n}) <1$ and hence $m_{C_{n,r}}[0,1) \geq k$.
    
    Now, we assume that $n=d(C_{n,r})+\lceil \frac{r}{2} \rceil \equiv 0 \pmod{3}$. Then $d(C_{n,r}) \equiv 1 \pmod{3}$.
    Suppose that $m_{C_{n,r}}[0,1)<k$. Then $\mu_k(C_{n,r}) \geq 1$.
    Since
    \begin{equation*}
        m_{P_n}[0,1) =\bigg\lceil \frac{n}{3} \bigg\rceil=\left\lceil \frac{d(C_{n,r})+\lceil \frac{r}{2} \rceil}{3} \right\rceil=\left\lceil \frac{d(C_{n,r})}{3} \right\rceil+\bigg\lceil \frac{r}{6} \bigg\rceil-1=k,
    \end{equation*}
    we have $\mu_{k+1}(P_n)=1$ and hence $\mu_k(C_{n,r})=1$.
    Note that $L(C_{n,r})-L(P_n)=N=(n_{ij})$, where
    \begin{equation*}
        n_{ij}=\left\{\begin{array}{rl}
            1, & \mbox{if $(i,j) \in \{(1,1),~(r,r)\}$,}\\
            -1, & \mbox{if $(i,j) \in \{(1,r),~(r,1)\}$,}\\
            0, & \mbox{otherwise.}
        \end{array}\right.
    \end{equation*}
     Since $N$ is permutationally similar to $L(K_2 \cup (n-2)K_1)$, we have 
    \begin{equation*}
        \mu_{1}(N)=\cdots=\mu_{n-1}(N)=0~~\text{and}~~\mu_n(N)=2.
    \end{equation*}
    Then $\mu_k(C_{n,r})=\mu_{k+1}(P_n)+\mu_{n-1}(N)$.
    By Lemma~\ref{lem:eig_sum}, there exists an eigenvector $x$ such that $$L(C_{n,r})x=\mu_k(C_{n,r})x=x,$$ $$L(P_n)x=\mu_{k+1}(P_n)x=x,~~\text{and}$$ $$Nx=\mu_{n-1}(N)x=0.$$
    Let $x$ be the eigenvector corresponding to the Laplacian eigenvalue 1 of $P_n$, which is given in Lemma~\ref{lem:path-one}.
    Then the first coordinate of $x$ is 1.
    Since $r \equiv 3,4 \pmod{6}$, the $r$th coordinate of $x$ is $-1$.
    Hence $Nx \neq 0$. This is a contradiction and hence $m_{C_{n,r}}[0,1) \geq k$.

    {\bf Case 4.} Suppose that $r \equiv 0, 5 \pmod{6}$. Then $\lceil \frac{r}{2} \rceil\equiv 0 \pmod{3}$. Since
    \begin{equation*}
        m_{P_n}[0,1)=\bigg\lceil \frac{n}{3} \bigg\rceil=\left\lceil \frac{d(C_{n,r})+\lceil \frac{r}{2} \rceil}{3} \right\rceil=\left\lceil \frac{d(C_{n,r})}{3} \right\rceil+\bigg\lceil \frac{r}{6} \bigg\rceil=k+1,
    \end{equation*}
     we obtain $\mu_k(C_{n,r})\leq \mu_{k+1}(P_{n}) <1$. 
     Thus $m_{C_{n,r}}[0,1) \geq k$.
\end{proof}

\begin{figure}
    \centering
    \includegraphics[width=16cm]{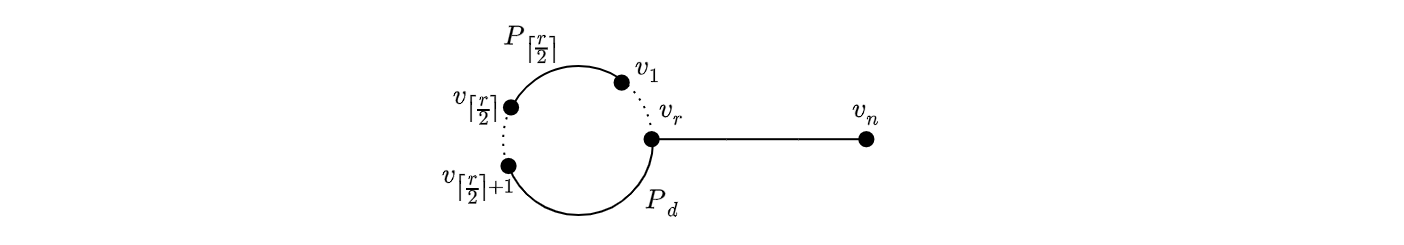}
    \caption{The subgraph $P_d\cup P_{\lceil\frac{r}{2}\rceil}$ of $C_{n,r}$}
    \label{fig:subgraph}
\end{figure}
The following proposition shows that for the special case $d(C_{n,r}) \equiv 0 \pmod{3}$ and $r \not\equiv 0 \pmod{6}$, we can find the exact number of Laplacian eigenvalues in the interval $[0,1)$.
\begin{proposition}\label{prop:lollipop-3}
    Let $C_{n,r}$ be the lollipop graph with $n$ vertices and girth $r$.
    If $d(C_{n,r})\equiv 0\pmod{3}$ and $r \not\equiv 0 \pmod{6}$ then 
    \begin{equation*}
        m_{C_{n,r}}[0,1) = \frac{d(C_{n,r})}{3} +\bigg\lceil \frac{r}{6} \bigg\rceil.
    \end{equation*}
\end{proposition}
\begin{proof}
    For simplicity of notation, we write $d$ instead of $d(C_{n,r})$.
    Let $k=\frac{d}{3} +\lceil \frac{r}{6} \rceil$.
    Then 
    $$k=\frac{d}{3} +\bigg\lceil \frac{r}{6} \bigg\rceil =\left\lceil\frac{d+\lceil \frac{r}{2} \rceil}{3}\right\rceil=\bigg\lceil \frac{n}{3} \bigg\rceil.$$
    By deleting edges $v_1v_r$ and $v_{\lceil \frac{r}{2} \rceil}v_{\lceil \frac{r}{2} \rceil+1}$ from $G$,
    we obtain a subgraph $P_d \cup P_{\lceil \frac{r}{2} \rceil}$ (see Figure~\ref{fig:subgraph}).
    Note that $m_{P_d \cup P_{\lceil \frac{r}{2} \rceil}}[0,1)=k$.
    Since $d \equiv 0 \pmod{3}$, we obtain $\mu_{k+1}(P_d \cup P_{\lceil \frac{r}{2} \rceil})=1$.
    By Lemma~\ref{lem:interlacing}, we have
    \begin{equation*}
        \mu_{k-1}(C_{n,r}) \leq \mu_k(P_n) \leq \mu_{k}(C_{n,r}) \leq \mu_{k+1}(P_n)\leq \mu_{k+1}(C_{n,r})
    \end{equation*}
    and
    \begin{equation*}
        \mu_k(P_n) \leq \mu_{k+1}(P_d \cup P_{\lceil \frac{r}{2} \rceil}) \leq \mu_{k+1}(P_n) \leq \mu_{k+2}(P_d \cup P_{\lceil \frac{r}{2} \rceil}).
    \end{equation*}  
    In order to prove $m_{C_{n,r}}[0,1) =k$, it suffices to show that 
    \begin{equation*}
        %\mu_k(P_n) \leq 
        \mu_{k}(C_{n,r}) < \mu_{k+1}(P_d \cup P_{\lceil \frac{r}{2} \rceil}) =1. %\leq \mu_{k+1}(P_n).
    \end{equation*}
    If $\mu_k(P_n)=\mu_k(C_{n,r})$, then $\mu_k(C_{n,r})=\mu_k(P_n)<1$ since $m_{P_n}[0,1)=k$.
    Suppose that $\mu_k(P_n)\neq \mu_k(C_{n,r})$.
    Let $\alpha$ be a real number such that 
    \begin{equation*}
        \mu_{k-1}(C_{n,r}) < \alpha < \mu_k(C_{n,r}).
    \end{equation*}
    If the signs of the values $\phi(C_{n,r};\alpha)$ and $\phi(C_{n,r};1)$ are opposite, then $$\alpha <\mu_k(C_{n,r}) < \mu_{k+1}(P_d \cup P_{\lceil \frac{r}{2}\rceil})=1.$$
    Thus we only need to show that the signs of the values $\phi(C_{n,r};\alpha)$ and $\phi(C_{n,r};1)$ are opposite.
    First, we determine the sign of the value $\phi(C_{n,r};\alpha)$. 
    Suppose that $r\equiv 1,2 \pmod{6}$. Then $n=d+\lceil \frac{r}{2} \rceil \equiv 1 \pmod{3}$.
    Let $n=3t+1$ for some positive integer $t$.
    Then $k=\lceil \frac{n}{3} \rceil=t+1$.
    Hence $n$ and $k$ are either both even or both odd.
    By considering the shape of the graph of $y=\phi(C_{n,r};x)$, we deduce that the value of $\phi(C_{n,r};\alpha)$ is negative for 
    $\mu_{k-1}(C_{n,r})<\alpha < \mu_k(C_{n,r})$.
    Similarly, if $r\equiv 5 \pmod{6}$, then the value of $\phi(C_{n,r};\alpha)$ is negative.
    Suppose that $r\equiv 3,4 \pmod{6}$. 
    Then $n=d+\lceil \frac{r}{2} \rceil \equiv 2 \pmod{3}$.
    Let $n=3t+2$ for some positive integer $t$.
    Then $k=\lceil \frac{n}{3} \rceil=t+1$.
    Thus, $n$ and $k$ have different parity.
    By considering the shape of the graph of $y=\phi(C_{n,r};x)$, we deduce that the value of $\phi(C_{n,r};\alpha)$ is positive for 
    $\mu_{k-1}(C_{n,r})<\alpha < \mu_k(C_{n,r})$.
    
    Now, we determine the sign of the value $\phi(C_r;1)$.
    By Lemma  \ref{lem:char-PC} (c) and Lemma \ref{lem:char-lollipop},
    \begin{equation*}
        \phi(C_{n,r};1)=-\big(\phi(P_{r+1};1)-\phi(P_{r-1};1)+2(-1)^{r+1}\big)\phi(P_{n-r-1};1)-\phi(P_{n-r};1)\phi(P_r;1).
    \end{equation*}
    Since  $\phi(P_0)=0$, $\phi(P_1;1)=1$ and
    $\phi(P_{n+1})=(x-2)\phi(P_n)-\phi(P_{n-1})$,
    we have
    \begin{equation*}
        \phi(P_n;1)=\left\{\begin{array}{rl}
            0, & \mbox{if } n\equiv 0 \pmod{3},\\
            1, & \mbox{if } n\equiv 1 \pmod{3},\\
            -1, & \mbox{if } n\equiv 2 \pmod{3}.
        \end{array}\right.
    \end{equation*}
    Since $n-r=d-\lfloor \frac{r}{2} \rfloor\equiv -\lfloor \frac{r}{2} \rfloor \pmod{3}$, we have
    \begin{multline*}
        \phi(C_{n,r};1)=-\big( \phi(P_{r+1};1)-\phi(P_{r-1};1)+2(-1)^{r+1} \big)\phi(P_{-\lfloor \frac{r}{2} \rfloor-1};1)-\phi(P_{-\lfloor \frac{r}{2} \rfloor};1)\phi(P_r;1).
    \end{multline*}
    Then
    \begin{equation*}
        \phi(C_{n,r};1)=\left\{\begin{array}{rl}
            1, & \mbox{if } r\equiv 1 \pmod{6},\\
            2, & \mbox{if } r\equiv 2 \pmod{6},\\
            -4, & \mbox{if } r\equiv 3 \pmod{6},\\
            -1, & \mbox{if } r\equiv 4 \pmod{6},\\
            1, & \mbox{if } r\equiv 5 \pmod{6}.
         \end{array}\right.
    \end{equation*}
    Thus the signs of the values $\phi(C_{n,r};1)$ are positive when $r\equiv 1,2,5 \pmod{6}$ and negative when $r \equiv 3,4 \pmod{6}$. 
    Hence $\phi(C_{n,r};\alpha)$ and $\phi(C_{n,r};1)$ have opposite signs. 
\end{proof}

\begin{remark}
    In Proposition~\ref{prop:lollipop-3}, the lower bound is improved by 1.
    We can rephrase Propositions~\ref{prop:lollipop} and \ref{prop:lollipop-3} as follows:\\
    If $r \not\equiv 0 \pmod{6}$, then
    \begin{equation*}
        m_{C_{n,r}}[0,1) \geq \bigg\lceil\frac{d(C_{n,r})+1}{3}\bigg\rceil+\bigg\lceil \frac{r}{6} \bigg\rceil-1.
    \end{equation*}
    If $r \equiv 0 \pmod{6}$, then
    \begin{equation*}
        m_{C_{n,r}}[0,1) \geq \bigg\lceil\frac{d(C_{n,r})}{3}\bigg\rceil+\bigg\lceil \frac{r}{6} \bigg\rceil-1.
    \end{equation*}
\end{remark}

\begin{lemma}\label{lem:cylce-one}
    Let $C_n$ be the cycle graph with $n$ vertices and let  $x=\begin{bmatrix} x_1 & \cdots & x_n \end{bmatrix}^T$ and $y=\begin{bmatrix} y_1 & \cdots & y_n \end{bmatrix}^T$  be vectors of order $n$, where    
    \begin{equation*}
        x_i =\left\{\begin{array}{rl}
                1, & \mbox{if $i\equiv 1,~ 6 \pmod{6}$,}\\
                0, & \mbox{if $i\equiv 2,~ 5 \pmod{6}$,}\\
                -1, & \mbox{if $i\equiv 3,~ 4 \pmod{6}$}
             \end{array}\right.
             \quad \text{and} \quad
         y_i =\left\{\begin{array}{rl}
                1, & \mbox{if $i\equiv 1,~ 2 \pmod{6}$,}\\
                0, & \mbox{if $i\equiv 3,~ 6 \pmod{6}$,}\\
                -1, & \mbox{if $i\equiv 4,~ 5 \pmod{6}$.}
             \end{array}\right.
    \end{equation*}
    If $n$ is divisible by 6, then the
    vectors $x$ and $y$ are eigenvectors corresponding to the Laplacian eigenvalue $1$ of $C_n$.
\end{lemma}
\begin{proof}
  It  can be verified by direct computation.  
\end{proof}

The following proposition provides the conditions under which a lollipop graph has 1 as a Laplacian eigenvalue.
\begin{proposition}\label{prop:r=01-lollipop}
    Let $C_{n,r}$ be a lollipop graph with $n$ vertices and girth $r$.
    Let $x$ be an eigenvector corresponding to the Laplacian eigenvalue 1 of $P_n$ in Lemma~\ref{lem:path-one}
    and let $y$ be an eigenvector corresponding to the Laplacian eigenvalue 1 of $C_r$ in Lemma~\ref{lem:cylce-one}.
    Then 
    \begin{itemize}
        \item[(a)] If $r \equiv 0 \pmod{6}$, then 1 is a Laplacian eigenvalue of $C_{n,r}$ with the corresponding eigenvector $z$, where the vector $z$ has coordinates
    \begin{equation*}
        z_i=\left\{ 
            \begin{array}{rl}
                y_i, & \mbox{if $i=1,\ldots,r$}, \\
                0, & \mbox{if $i=r+1,\ldots,n$.}
            \end{array}
        \right.
    \end{equation*}

        Moreover, 
        \begin{equation*}
            m_{C_{n,r}}(1)=\left\{ 
            \begin{array}{rl}
                2, & \mbox{if  $n \equiv 0 \pmod{3}$,} \\
                1, & \mbox{if  $n \not\equiv 0 \pmod{3}$.}
            \end{array}\right.
        \end{equation*}
        \item[(b)] If $r \equiv 1 \pmod{6}$ and $n \equiv 0 \pmod{3}$, then  1 is a Laplacian eigenvalue of $C_{n,r}$ with the corresponding eigenvector $x$ and $m_{C_{n,r}}(1)=1$.
        \item[(c)] If $r\equiv 3 \pmod{6}$ and $n \equiv 1 \pmod{3}$, then 1 is a Laplacian eigenvalue of $C_{n,r}$ with the corresponding eigenvector $w$, where the vector $w$ has coordinates
        \begin{equation*}
        w_i=\left\{ 
            \begin{array}{cl}
                ~\,y_i, & \mbox{if $i=1,\ldots,r$}, \\
                -2y_{i-r}, & \mbox{if $i=r+1,\ldots,n$.}
            \end{array}
        \right.
    \end{equation*}
        Moreover, $m_{C_{n,r}}(1)=1$. 
   \end{itemize}
\end{proposition}
\begin{proof}
    (a) Suppose that $r \equiv 0 \pmod{6}$.
    Note that $L(C_{n,r})-L(C_r \cup P_{n-r})=M=(m_{ij})$, where
    \begin{equation*}
        m_{ij}=\left\{\begin{array}{rl}
            1, & \mbox{if $(i,j) \in \{(r,r),~(r+1,r+1)\}$,}\\
            -1, & \mbox{if $(i,j) \in \{(r,r+1),~(r+1,r)\}$,}\\
            0, & \mbox{otherwise.}
        \end{array}\right.
    \end{equation*}
    Since $r \equiv 0 \pmod{6}$, the coordinate $z_r$ of the vector $z$ is 0 and $L(C_r)y=y$.
    Since $z_r=z_{r+1}=0$, we obtain $Mz=0$.
    Then $$L(C_{n,r})z=L(C_r \cup P_{n-r})z+Mz=z.$$ 
    Thus 1 is a Laplacian eigenvalue with the corresponding eigenvector $z$.
    
    Suppose that $n \equiv 0 \pmod{3}$ and $r \equiv 0 \pmod{6}$. Then $n-r \equiv 0 \pmod{3}$. Thus, by Corollary~\ref{coro:path}, $m_{C_{n,r}}(1)=m_{C_r}(1)=2$.
    Now, we claim that $m_{C_{n,r}}(1)=1$ if $r \equiv 0 \pmod{6}$ and $n \not\equiv 0 \pmod{3}$. Suppose that $m=m_{C_{n,r}}(1)>1$.
    Then $$\mu_k(C_{n,r})=\cdots=\mu_{k+m-1}(C_{n,r})=1$$ for some $k$.
    Since $P_n$ is a subgraph of $C_{n,r}$,
    by Lemma~\ref{lem:interlacing}, 
    \begin{equation*}
        1=\mu_{k}(C_{n,r}) \leq \mu_{k+1}(P_n) \leq \mu_{k+1}(C_{n,r})=1.
    \end{equation*}
    Thus $\mu_{k+1}(P_n)=1$, which contradicts to $m_{P_n}(1)=0$.
    Hence $m_{C_{n,r}}(1)=1$.
    
    (b) Suppose that $r \equiv 1 \pmod{6}$ and $n\equiv 0 \pmod{3}$.
    Note that $L(C_{n,r})-L(P_n)=N=(n_{ij})$, where
    \begin{equation*}
        n_{ij}=\left\{\begin{array}{rl}
            1, & \mbox{if $(i,j) \in \{(1,1),~(r,r)\}$,}\\
            -1, & \mbox{if $(i,j) \in \{(1,r),~(r,1)\}$,}\\
            0, & \mbox{otherwise.}
        \end{array}\right.
    \end{equation*}
    Since $r \equiv 1 \pmod{6}$, we have $x_r=1$.
    Since $x_1=x_r=1$, we obtain $Nx=0$.
    Since $n\equiv 0 \pmod{3}$, we have $L(P_n)x=x$.
    Then $$L(C_{n,r})x=L(P_n)x+Nx=x.$$
    Thus 1 is a Laplacian eigenvalue with the corresponding eigenvector $x$.
    
    Now, we show that $m_{C_{n,r}}(1)=1$.
    We consider the subgraph $C_{r} \cup P_{n-r}$ of $C_{n,r}$. Since $r \equiv 1 \pmod{6}$ and $n\equiv 0 \pmod{3}$, we have $n-r \not\equiv 0 \pmod{3}$. 
    Thus neither $C_r$ nor $P_{n-r}$ have the Laplacian eigenvalue 1. 
    Suppose that $m=m_{C_{n,r}}(1)>1$.
    Assume that $$\mu_k(C_{n,r})=\cdots=\mu_{k+m-1}(C_{n,r})=1$$ for some $k$.
    Since $C_{r} \cup P_{n-r}$ is a subgraph of $C_{n,r}$,
    by Lemma~\ref{lem:interlacing}, 
    \begin{equation*}
        1=\mu_{k}(C_{n,r}) \leq \mu_{k+1}(C_{r} \cup P_{n-r}) \leq \mu_{k+1}(C_{n,r})=1.
    \end{equation*}
    Thus $\mu_{k+1}(C_{r} \cup P_{n-r})=1$, which contradicts to $m_{C_{r} \cup P_{n-r}}(1)=0$. 
    Hence $m_{C_{n,r}}(1)=1$.

    (c) Let $l_{ij}$ be the $(i,j)$-entry of $L(C_{n,r})$. 
    To prove $L(C_{n,r})w=w$, we compute the coordinates of $L(C_{n,r})w$. 
    For $i\not\in \{1, r, n\}$, the entries of $L(C_{n,r})$ are
    \begin{equation*}
        l_{ij}=\left\{ 
            \begin{array}{rl}
                2, &  \text{if $j=i$,}\\
                -1, & \text{if $j=i-1, i+1$,}\\
                0, & \text{otherwise}.
            \end{array}
        \right.
    \end{equation*} 
    Since $\begin{bmatrix}
        w_{i-1} & w_i & w_{i+1}
    \end{bmatrix}^T$ is one of 
    \begin{equation*}
        c\begin{bmatrix}
            1 & 1 & 0
        \end{bmatrix}^T,~~
        c\begin{bmatrix}
            1 & 0 & -1
        \end{bmatrix}^T,~~\text{or}~~
        c\begin{bmatrix}
            0 & 1 & 1
        \end{bmatrix}^T,
    \end{equation*}
    where $c=\pm1,\pm2$, we obtain
    \begin{equation*}
        \begin{bmatrix}
            -1 & 2 & -1
        \end{bmatrix}
        \begin{bmatrix}
            w_{i-1} & w_i & w_{i+1}
        \end{bmatrix}^T=w_i.
    \end{equation*}
    Hence the $i$th coordinate of $L(C_{n,r})w$ is equal to $w_i$ for $i\not\in \{1,r,n\}$.
    Now, we examine the remaining cases $i\in\{1,r,n\}$. 
    Since  $w_1=w_2=1$, $w_r=0,$ and the entry of the first row of $L(C_{n,r})$ is
    \begin{equation*}
        l_{1j}=\left\{
            \begin{array}{rl}
                2, & \text{if $j=1$,}  \\
                -1, & \text{if $j=2$ and $r$,}\\
                0, & \text{otherwise,}
            \end{array}
        \right.
    \end{equation*}
    the first coordinate of $L(C_{n,r})w$ is $w_1=1$.
    Since $w_1=w_{r-1}=1$, $w_r=0$, $w_{r+1}=-2$ and  the entry of the $r$th row of $L(C_{n,r})$ is
    \begin{equation*}
        l_{rj}=\left\{
            \begin{array}{rl}
                3, & \text{if $j=r$,}  \\
                -1, & \text{if $j=1$, $r-1$ and $r+1$,}\\
                0, & \text{otherwise,}
            \end{array}
        \right.
    \end{equation*}
    the $r$th coordinate of $L(C_{n,r})w$ is $w_r=0$.
    Since $w_{n-1}=0$, $w_n=\pm 2$ and the entry of the last row of $L(C_{n,r})$ is
    \begin{equation*}
        l_{nj}=\left\{
            \begin{array}{rl}
                -1, & \text{if $j=n-1$,}  \\
                1, & \text{if $j=n$,}\\
                0, & \text{otherwise,}
            \end{array}
        \right.
    \end{equation*}
     the last coordinate of $L(C_{n,r})w$ is $w_n=\pm 2$.
    Thus $L(C_{n,r})w=w$.
    
    Now, we show that $m_{C_{n,r}}(1)=1$. 
    Suppose that $m=m_{C_{n,r}}(1)>1$.
    Then $$\mu_k(C_{n,r})=\cdots=\mu_{k+m-1}(C_{n,r})=1$$ for some $k$.
    Since $P_n$ is a subgraph of $C_{n,r}$,
    by Lemma~\ref{lem:interlacing}, 
    \begin{equation*}
        1=\mu_{k}(C_{n,r}) \leq \mu_{k+1}(P_n) \leq \mu_{k+1}(C_{n,r})=1.
    \end{equation*}
    Thus $\mu_{k+1}(P_n)=1$, which contradicts to $m_{P_n}(1)=0$.
    Hence $m_{C_{n,r}}(1)=1$.
\end{proof}
%============================================

% ======================================
\subsection{Compass graphs}
%=====================================
Finally, we determine the lower bound for $m_G [0,1)$, where $G$ is a compass graph in Figure~\ref{fig:types}. We label the $n$ vertices of the graph as shown in Figure ~\ref{fig:compass}.
\begin{proposition}\label{prop:compass_1}
    Let $C_{n,r}(r',t)$ be a graph with $n$ vertices and girth $r$. 
    Suppose that $n\equiv t+\lceil \frac{r}{2} \rceil  \pmod{3}$ or $t \equiv0 \pmod{3}$. Then 
    \begin{equation*}
        m_{C_{n,r}(r',t)}[0,1) \geq \bigg\lceil\frac{d(C_{n,r}(r',t))}{3}\bigg\rceil + \bigg\lceil \frac{r}{6} \bigg\rceil -1.
    \end{equation*}
\end{proposition}
\begin{proof}
    We set $G = C_{n,r}(r',t)$. 
    Note that $s=n-r-t$ and $r'=d(G)-t-s$.
    Let $k=\lceil\frac{d(G)}{3}\rceil + \left\lceil \frac{r}{6} \right\rceil -1$.
    A subgraph $P_t \cup C_{n-t,r}$ can be obtained from $G$ by deleting the edge $v_tv_{t+r'}$, where $v_t$ and $v_{t+r'}$ are labeled as in Figure~\ref{fig:compass}.
    % By Propositions \ref{prop:lollipop} and Lemma \ref{lem:path-one}, we have
    % \begin{equation*}
    %     m_{P_t \cup C_{n-t,r}}[0,1) \geq \left\lceil\frac{d(C_{{n-t},r})}{3}\right\rceil +\bigg\lceil \frac{r}{6} \bigg\rceil -1 +\left\lceil\frac{t}{3} \right\rceil.
    % \end{equation*}
    Let $\alpha=\lfloor \frac{r}{2} \rfloor - r'$. 
    Since $\lfloor \frac{r}{2} \rfloor \geq r'$, we have $\alpha \geq 0$ .
    Note that $d(C_{n-t,r})+t=d(G)+\alpha$. %$d(G)=n-(r-r')$ and $d(C_{{n-t},r})=n-t-\lceil \frac{r}{2} \rceil$ . %\not \equiv 0 \pmod{3}$.

    {\bf Case 1. }Suppose that $t \equiv 0 \pmod{3}$.
    By Corollary \ref{coro:path}, we have $m_{G}(1)=m_{C_{n-t,r}}(1)$. 
    Put $l=m_{C_{n-t,r}}(1)$. Then $m_{P_t \cup C_{n-t,r}}(1)=l+1$. 
    Note that 
    \begin{align*}
        m_{P_t \cup C_{n-t,r}}[0,1) \geq &\left\lceil\frac{d(C_{{n-t},r})}{3}\right\rceil +\bigg\lceil \frac{r}{6} \bigg\rceil -1 +\left\lceil\frac{t}{3} \right\rceil\\
        = & \left\lceil\frac{d(G)+\alpha}{3}\right\rceil +\bigg\lceil \frac{r}{6} \bigg\rceil -1\\
        \geq & \left\lceil\frac{d(G)}{3}\right\rceil +\bigg\lceil \frac{r}{6} \bigg\rceil -1=k.
    \end{align*}
    By Lemma~\ref{lem:interlacing}, we obtain $\mu_{k-1}(G) \leq \mu_{k}(P_t \cup C_{n-t,r}) <1$.
    Suppose that $\mu_k(G) \geq 1$. By Lemma \ref{lem:interlacing}, we have
    \begin{equation*}
        1\leq \mu_k(G) \leq \mu_{k+1}(P_t \cup C_{n-t,r})\leq \mu_{k+1}(G) \leq \cdots \leq \mu_{k+l}(G) \leq \mu_{k+l+1}(P_t \cup C_{n-t,r})=1.
    \end{equation*}
    It follows that 
    \begin{equation*}
        \mu_k(G) = \mu_{k+1}(G) =\cdots= \mu_{k+l}(G)=1.
    \end{equation*}
    Thus, we obtain $m_G(1)=l+1$. 
    This contradicts to $m_G(1)=l$.
    Thus, $\mu_k(G)<1$ and hence $m_G[0,1) \geq k$.
    
    {\bf Case 2. }Suppose that $d(C_{{n-t},r})=n-t-\lceil \frac{r}{2} \rceil \equiv0 \pmod{3}$.
    If $r\not \equiv 0 \pmod{6}$, by Proposition \ref{prop:lollipop-3}, we have 
    \begin{align*}
        m_{P_t \cup C_{n-t,r}}[0,1) \geq &\left\lceil\frac{d(C_{{n-t},r})}{3}\right\rceil +\bigg\lceil \frac{r}{6} \bigg\rceil +\left\lceil\frac{t}{3} \right\rceil\\
        = & \left\lceil\frac{d(G)+\alpha}{3}\right\rceil +\bigg\lceil \frac{r}{6} \bigg\rceil\\
        \geq & \left\lceil\frac{d(G)}{3}\right\rceil +\bigg\lceil \frac{r}{6} \bigg\rceil=k+1.
    \end{align*}
    Thus $\mu_k(G)\leq \mu_{k+1}(P_t \cup C_{n-t,r})<1$ and hence $m_G[0,1) \geq k$.
    If $r \equiv 0 \pmod{6}$, then $0 \equiv d(C_{{n-t},r}) =\lfloor \frac{r}{2} \rfloor +s \equiv s \pmod{3}$.
    In this case, we consider the subgraph $P_s \cup C_{n-s,r}$. 
    Since $s \equiv 0 \pmod{3}$, by Case 1, we obtain $m_G[0,1) \geq k$.
\end{proof}
\begin{lemma}\label{lem:comp-1}
    Let $C_{n,r}(r',t)$ be a graph with $n$ vertices and girth $r$. 
    Suppose that $r \equiv 0 \pmod{6}$ and $r'=\frac{r}{2}$.
    Then $C_{n,r}(r',t)$ has a Laplacian eigenvalue 1.
\end{lemma}
\begin{proof}
    Let $G=C_{n,r}(r',t)$. % and $s=n-r-t$.
    % Suppose that $t \equiv 0\pmod{3}$ or $s \equiv 0\pmod{3}$.
    % Without loss of generality, we assume that $t \equiv 0 \pmod{3}$.
    % We consider the subgraph $P_t \cup C_{n-t,r}$ obtained by removing the edge $v_tv_{t+r'}$.
    % Note that the lollipop graph  $C_{n-t,r}$ has a Laplacian eigenvalue 1, by Lemma~\ref{prop:r=01-lollipop}. 
    % If $t \equiv 0\pmod{3}$, then, by Corollary~\ref{coro:path}, the graph $G$ also has a Laplacian eigenvalue 1. 
    % Now, suppose that $t \not\equiv 0\pmod{3}$ and $s \not\equiv 0\pmod{3}$.
    and let $P_t \cup C_{n-t,r}$ be a subgraph of $G$ obtained by deleting the edge $v_tv_{t+r'}$, where $v_t$ and $v_{t+r'}$ are labeled as in Figure~\ref{fig:compass}.
    Note that $L(G)-L(P_t \cup C_{n-t,r})=A=(a_{ij})$, where
    \begin{equation*}
        a_{ij}=\left\{\begin{array}{rl}
            1, & \mbox{if $(i,j) \in \{(t,t),~(t+r',t+r')\}$,}\\
            -1, & \mbox{if $(i,j) \in \{(t,t+r'),~(t+r',t)\}$,}\\
            0, & \mbox{otherwise.}
        \end{array}\right.
    \end{equation*}
    Let $y$ be the eigenvector corresponding to the Laplacian eigenvalue 1 of $C_r$ in Lemma~\ref{lem:cylce-one} and let $z$ be a vector of order $n$ with coordinates 
    \begin{equation*}
        z_i=\left\{ 
        \begin{array}{rl}
            y_{i-t}, & \mbox{if $i=t+1,\ldots, t+r$,} \\
            0, & \mbox{otherwise.}
        \end{array}
        \right.
    \end{equation*}
    % such that the first $t$ coordinates and the last $n-r-t$ coordinates are 0 and $z_{t+i}=y_i$ for $i=1,\ldots, r$.
    Since $z_t=0$ and $z_{t+r'}=y_{r'}=0$, we have 
    $Az=0$.
    Then, by Proposition~\ref{prop:r=01-lollipop}~(a), 
    $$L(G)z=L(P_t \cup C_{n-t,r})z+Az=z.$$
    Thus $G$ has a Laplacian eigenvalue 1.
\end{proof}
\begin{figure}
    \centering
    \includegraphics[width=16cm]{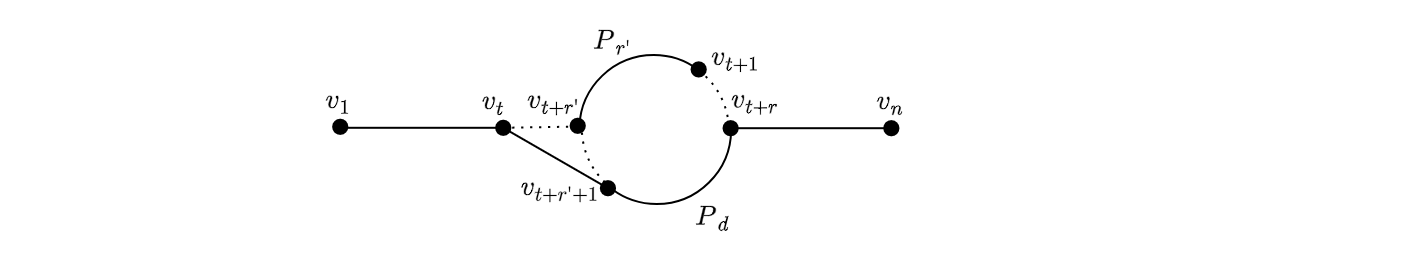}
    \caption{The graph $P_d \cup P_{r'}$ obtained from $C_{n,r}(r',t)$, where $d=d(C_{n,r}(r,t))$}
    \label{fig:com-path}
\end{figure}
%The following proposition will be used in the proof of Proposition~\ref{prop:compass_2}.
A lower bound for the compass graph is provided under specific conditions in the following.
\begin{proposition}\label{prop:comp-r6}
     Let $C_{n,r}(r',t)$ be a graph with $n$ vertices and girth $r$.
     Suppose that $n \equiv 0 \pmod{3}$, $r \equiv 0 \pmod{6}$, $r'=\frac{r}{2}$, and $t \equiv 1 \pmod{3}$.
     Then $$m_{C_{n,r}(r',t)}[0,1) \geq \bigg\lceil\frac{d(C_{n,r}(r',t))}{3}\bigg\rceil + \bigg\lceil \frac{r}{6} \bigg\rceil.$$
\end{proposition}
\begin{proof}
    To simplify the notation, put $G=C_{n,r}(r',t)$ and $d=d(C_{n,r}(r',t))$. Let $k=\big\lceil\frac{d}{3}\big\rceil + \big\lceil \frac{r}{6} \big\rceil$.
    First, we consider a subgraph $P_t \cup C_{n-t,r}$ obtained from $G$ by removing the edge $v_tv_{t+r'}$. 
    Note that $t \equiv 1 \pmod{3}$ and 
    $d(C_{n-t,r})=n-t-\lceil\frac{r}{2}\rceil \equiv 2\pmod{3}$.
    Then
    \begin{equation*}
        m_{P_t \cup C_{n-t,r}}[0,1) \geq \left\lceil\frac{d(C_{{n-t},r})}{3}\right\rceil +\bigg\lceil \frac{r}{6} \bigg\rceil -1+\left\lceil\frac{t}{3} \right\rceil
        = \left\lceil\frac{d(G)}{3}\right\rceil +\bigg\lceil \frac{r}{6} \bigg\rceil=k.
    \end{equation*}
    By Lemma~\ref{lem:interlacing}, we have $\mu_{k-1}(G) \leq \mu_k(P_t \cup C_{n-t,r})<1$ and hence $m_G[0,1)\geq k-1$.
    Suppose that $\mu_{k}(G)\geq1$.
    Since the graph $P_t \cup C_{n-t,r}$ has the Laplacian eigenvalue 1, by Lemma~\ref{lem:comp-1},
    we obtain $\mu_{k}(G)=1$.
    Now, we consider the graph $P_d \cup P_{r'}$ which is obtained from $G$ by deleting three edges $v_tv_{t+r'}$, $v_{t+1}v_{t+r}$, and $v_{t+r'}v_{t+r'+1}$ and adding an edge $v_tv_{t+r'+1}$ (see Figure~\ref{fig:com-path}).
    Since $d \equiv 0 \pmod{3}$ and $r' \equiv 0 \pmod{3}$, we have 
    \begin{equation*}
        \mu_{k+1}(P_d \cup P_{r'})=\mu_{k+2}(P_d \cup P_{r'})=1.
    \end{equation*}
    Note that $L(G)-L(P_d \cup P_{r'})=B=(b_{ij})$, where
    \begin{equation*}
         b_{ij}=\left\{\begin{array}{rl}
            2, & \mbox{if $(i,j) \in \{(t+r',t+r')\}$,}\\
            1, & \mbox{if $(i,j) \in \{(t,t+r'+1),(t+1,t+1),(t+r'+1,t),(t+r,t+r)\}$,}\\
            -1, & \mbox{if $(i,j) \in \{(t,t+r'),(t+1,t+r),(t+r',t),(t+r',t+r'+1),$}\\
             &\qquad\qquad \mbox{~~$(t+r'+1,t+r'),(t+r,t+1)\},$}\\
            0, & \mbox{otherwise.}
        \end{array}\right.
    \end{equation*}
    Since $B$ is permutationally similar to
    \begin{equation*}
        \left[\begin{array}{rrrrr}
            0 & 0 & -1 & 1 & 0\\
            0 & 1 & 0 & 0 & -1\\
            -1 & 0 & 2 & -1 & 0\\
            1 & 0 & -1 & 0 & 0\\
            0 & -1 & 0 & 0 & 1
        \end{array}\right] \oplus O_{n-5},
    \end{equation*}
    where $O_{n-5}$ is the square zero matrix of order $n-5$, we have
    \begin{equation*}
        \mu_1(B)=-1,~\mu_2(B)=\cdots=\mu_{n-2}(B)=0,~ \mu_{n-1}(B)=2,~~\text{and}~~\mu_n(B)=3.
    \end{equation*}
     Thus $\mu_k(G)=\mu_{k+2}(P_d\cup P_{r'})+\mu_{n-2}(B)$.
     By Lemma~\ref{lem:eig_sum}, there exists an eigenvector $z$ such that $$L(G)z=\mu_k(G)z=z,$$ $$L(P_d\cup P_{r'})z=\mu_{k+2}(P_d\cup P_{r'})z=z,~~\text{and}$$ $$Bz=\mu_{n-2}(B)z=0.$$
     Now, we show that there is no common eigenvector satisfying the above three equations.
     Let $x$ and $x'$ be the eigenvector corresponding to the Laplacian eigenvalue 1 of $P_d$ and $P_{r'}$, respectively.
   We define two vectors $z$ and $z'$ of order $n$ such that 
    \begin{equation*}
        z_i=\left\{\begin{array}{cl}
            x_i, &  \mbox{if $i \in \{1,\ldots,t\}$},\\
            0, &  \mbox{if $i\in \{t+1,\ldots,t+r'\}$},   \\
            x_{i-r'}, & \mbox{if $i \in\{t+r'+1,\ldots,n\}$}
        \end{array}
        \right.
        ~~\text{and}~~
        z'_i=\left\{\begin{array}{cl}
            x'_{i-t}, &  \mbox{if $i\in \{t+1,\ldots,t+r'\}$},   \\
            0, & \mbox{otherwise}.
        \end{array}
        \right.
    \end{equation*}
    Then $z$ and $z'$ are eigenvectors corresponding to the Laplacian eigenvalue 1 of $P_d \cup P_{r'}$.
    First, we show that $Bz \neq 0$.
    Note that 
    \begin{equation*}
        z_t=x_t,~z_{t+1}=z_{t+r'}=0, ~ z_{t+r'+1}=x_{t+1},~\text{and}~z_{t+r}=x_{t+r'}.
    \end{equation*}
    Suppose that $r=6l$ for some odd integer $l$. Then $r' \equiv 3 \pmod{6}$.
    If $t \equiv 1 \pmod{6}$, then
    $$z_t=1,~z_{t+r'+1}=0 ~~\text{and}~~ z_{t+r}=-1.$$
    If $t \equiv 4 \pmod{6}$, then
    $$z_t=-1,~ z_{t+r'+1}=0 ~~\text{and}~~ z_{t+r}=1.$$
    Suppose that $r=6l'$ for some even integer $l'$. Then $r' \equiv 0 \pmod{6}$.
    If $t \equiv 1 \pmod{6}$, then
    $$z_t=1,~z_{t+r'+1}=0  ~~\text{and}~~ z_{t+r}=1.$$
    If $t \equiv 4 \pmod{6}$, then
    $$z_t=-1,~z_{t+r'+1}=0  ~~\text{and}~~ z_{t+r}=-1.$$
    In each case, we obtain $Bz \neq 0$.
    Now, we prove that $Bz' \neq 0$.
    Note that 
    \begin{equation*}
        z'_t =0,~z'_{t+1}=x'_1=1,~ z'_{t+r'}=x'_{r'},~\text{and}~ z'_{t+r'+1} =z'_{t+r}=0.
    \end{equation*}
    Since 
    \begin{equation*}
        z'_{t+r'}=\left\{\begin{array}{rl}
            1 & \mbox{if $r' \equiv 0 \pmod{6}$,} \\
            -1 & \mbox{if $r' \equiv 3 \pmod{6}$,}
        \end{array}\right.
    \end{equation*}
    % If $r=6l$ for some odd integer $l$, then $z'_{t+r'}=-1$.
    % If $r=6l$ for some even integer $l$, then $z'_{t+r'}=1$.
    we have
    $Bz' \neq0$. This is a contradiction to $\mu_k(G)\geq 1$.
    Hence $m_G[0,1) \geq k$.
\end{proof}

\begin{proposition}\label{prop:compass_2}
    Let $C_{n,r}(r',t)$ be a graph with $n$ vertices and girth $r$. 
    Suppose that $n\not\equiv t+\lceil \frac{r}{2} \rceil  \pmod{3}$ and $t\not\equiv 0 \pmod{3}$. 
    Then 
    \begin{equation*}
        m_{C_{n,r}(r',t)}[0,1) \geq \bigg\lceil\frac{d(C_{n,r}(r',t))}{3}\bigg\rceil + \bigg\lceil \frac{r}{6} \bigg\rceil -1.
    \end{equation*}
\end{proposition}
\begin{proof}
    We set $G = C_{n,r}(r',t)$. 
    Note that $s=n-r-t$ and $r'=d(G)-t-s$.
    Let $k=\lceil\frac{d(G)}{3}\rceil + \left\lceil \frac{r}{6} \right\rceil -1$.
    A subgraph $P_t \cup C_{n-t,r}$ can be obtained from $G$ by deleting the edge $v_tv_{t+r'}$.
    Then $d(C_{{n-t},r})=n-t-\lceil \frac{r}{2} \rceil$.
    % By Propositions \ref{prop:lollipop} and Lemma \ref{lem:path-one}, we have
    % \begin{equation*}
    %     m_{P_t \cup C_{n-t,r}}[0,1) \geq \left\lceil\frac{d(C_{{n-t},r})}{3}\right\rceil +\bigg\lceil \frac{r}{6} \bigg\rceil -1 +\left\lceil\frac{t}{3} \right\rceil.
    % \end{equation*}
    Let $\alpha=\lfloor \frac{r}{2} \rfloor - r'$. 
    Since $\lfloor \frac{r}{2} \rfloor \geq r'$, we have $\alpha \geq 0$ .
    Note that $d(C_{n-t,r})+t=d(G)+\alpha$. %$d(G)=n-(r-r')$ and $d(C_{{n-t},r})=n-t-\lceil \frac{r}{2} \rceil$ . %\not \equiv 0 \pmod{3}$.
    We consider the following two cases:
    \begin{itemize}
        \item[(1)] $d(C_{{n-t},r}) \not\equiv 2 \pmod{3}$ or $t\not\equiv 2 \pmod{3}$.
        \item[(2)] $d(C_{{n-t},r}) \equiv 2 \pmod{3}$ and $t \equiv 2 \pmod{3}$.
    \end{itemize}

    {\bf Case 1. } Suppose that $d(C_{{n-t},r}) \not\equiv 2 \pmod{3}$ or $t\not\equiv 2 \pmod{3}$.
    Then \begin{align*}
        m_{C_{{n-t},r} \cup P_t}[0,1) \geq &\left\lceil\frac{d(C_{{n-t},r})}{3}\right\rceil +\bigg\lceil \frac{r}{6} \bigg\rceil -1 +\left\lceil\frac{t}{3} \right\rceil\\
        > & \left\lceil\frac{d(G)+\alpha}{3}\right\rceil +\bigg\lceil \frac{r}{6} \bigg\rceil -1\\
        \geq & \left\lceil\frac{d(G)}{3}\right\rceil +\bigg\lceil \frac{r}{6} \bigg\rceil -1=k.
    \end{align*}
    Thus, by Lemma~\ref{lem:interlacing}, we obtain $\mu_k(G)\leq \mu_{k+1}(P_t \cup C_{n-t,r})<1$ and hence $m_G[0,1) \geq k$.

    {\bf Case 2. } Suppose that $d(C_{{n-t},r}) \equiv 2 \pmod{3}$ and $t\equiv 2 \pmod{3}$.
    If $s\not \equiv 2 \pmod{3}$, then we consider the subgraph $P_s \cup C_{n-s,r}$. 
    Thus, by Case 1, we have $m_G[0,1) \geq k$. 
    Assume that $s \equiv 2 \pmod{3}$.
    For $\alpha \geq 3$, we have \begin{equation*}
        m_{P_t \cup C_{n-t,r}}[0,1) \geq \left\lceil\frac{d(G)+\alpha}{3}\right\rceil +\bigg\lceil \frac{r}{6} \bigg\rceil -1
        \geq \left\lceil\frac{d(G)}{3}\right\rceil +\bigg\lceil \frac{r}{6} \bigg\rceil=k+1.
    \end{equation*}
    By Lemma~\ref{lem:interlacing}, $\mu_k(G) \leq \mu_{k+1}(P_t \cup C_{n-t,r})<1$.
    Thus $m_G[0,1) \geq k$ for $\alpha \geq 3$.
    Now, we assume that $0 \leq \alpha <3$.

    \textbf{Case 2-1. } Suppose that $d(G) \not \equiv 1 \pmod{3}$.
    Let $G'$ be the graph obtained from $G$ by deleting the vertex $v_1$ in Figure~\ref{fig:compass}.
    Then, by Case 1, we have
    \begin{equation*}
        m_{G'}[0,1)\geq\left\lceil\frac{d(G)-1}{3}\right\rceil +\bigg\lceil \frac{r}{6} \bigg\rceil -1=\left\lceil\frac{d(G)}{3}\right\rceil +\bigg\lceil \frac{r}{6} \bigg\rceil -1.
    \end{equation*}
    Thus, by Lemma~\ref{lem:rem-pendant}, we obtain 
    \begin{equation*}
        m_{G}[0,1)\geq \left\lceil\frac{d(G)}{3}\right\rceil +\bigg\lceil \frac{r}{6} \bigg\rceil -1.
    \end{equation*}

    \textbf{Case 2-2. }
    Suppose that $d(G)\equiv 1 \pmod{3}$.
    Since $d(C_{{n-t},r})=\lfloor \frac{r}{2} \rfloor+s \equiv 2 \pmod{3}$, we have $\lfloor\frac{r}{2} \rfloor \equiv 0 \pmod{3}$, that is, $r \equiv 0,1 \pmod{6}$.
    Since $$d(G)=r'+s+t \equiv 1 \pmod{3}, t\equiv 2\pmod{3} ~~\text{and}~~ s\equiv 2\pmod{3},$$ we obtain $r' \equiv 0 \pmod{3}$.
    Thus, $\alpha=\lfloor\frac{r}{2}\rfloor-r'$ must be zero. 
    Hence the graph $G$ satisfies the following conditions: \begin{equation*}
        r \equiv 0,1\pmod{6},~r'=\bigg\lfloor \frac{r}{2}\bigg \rfloor,~~\text{and}~~t \equiv s\equiv 2\pmod{3}.
    \end{equation*}

    Suppose that $r \equiv 0 \pmod{6}$.
    Since the subgraph $G'$ is the compass graph $C_{n-1,r}(r',t-1)$, by Proposition~\ref{prop:comp-r6}, we have
    \begin{equation*}
        m_{G'}[0,1)\geq \bigg\lceil \frac{d(G)-1}{3} \bigg\rceil +\bigg\lceil \frac{r}{6} \bigg\rceil = \bigg\lceil \frac{d(G)}{3} \bigg\rceil +\bigg\lceil \frac{r}{6} \bigg\rceil-1.
    \end{equation*}
    Hence $m_G[0,1) \geq k$, by Lemma~\ref{lem:rem-pendant}.

    Suppose that $r \equiv 1 \pmod{6}$.
    Since $m_{P_t \cup C_{n-t,r}}[0,1) \geq k$, we have $$m_G[0,1) \geq k-1.$$
    Suppose that $\mu_k(G)\geq 1$.
    Since $n=r+t+s \equiv 0 \pmod{3}$, by Proposition~\ref{prop:r=01-lollipop}~(b), the lollipop $C_{n-t,r}$ has the Laplacian eigenvalue 1 with multiplicity 1.
    Thus, by Lemma~\ref{lem:interlacing}, we have $\mu_k(G)=\mu_{k+1}(P_t \cup C_{n-t,r})=1$. 
    Note that $L(G)-L(P_t \cup C_{n-t,r})=A=(a_{ij})$, where
    \begin{equation*}
        a_{ij}=\left\{\begin{array}{rl}
            1, & \mbox{if $(i,j) \in \{(t,t),~(t+ r',t+ r')\}$,}\\
            -1, & \mbox{if $(i,j) \in \{(t,t+ r'),~(t+ r',t)\}$,}\\
            0, & \mbox{otherwise.}
        \end{array}\right.
    \end{equation*}
    Since $A$ is permutationally similar to $L(K_2 \cup (n-2)K_1)$, we have
    \begin{equation*}
        \mu_{1}(A)=\cdots=\mu_{n-1}(A)=0~~\text{and}~~\mu_n(A)=2.
    \end{equation*}
    Then $\mu_k(G)=\mu_{k+1}(P_t \cup C_{n-t,r})+\mu_{n-1}(A)$. 
    By Lemma~\ref{lem:eig_sum}, there exists an eigenvector $z$ such that 
    $$L(G)z=\mu_k(G)z=z,$$ $$L(P_t \cup C_{n-t,r})z=\mu_{k+1}(P_t \cup C_{n-t,r})z=z,~~\text{and}$$ $$Az=\mu_{n-1}(A)z=0.$$
    %Since $t \equiv 2 \pmod{3}$, the path graph $P_t$ has no Laplacian eigenvalue 1.
    Then $z$ must be a vector of the form
    \begin{equation*}
        z=\left[\begin{array}{c}
             0  \\\hline
             x
        \end{array}\right],
    \end{equation*}
    where the first $t$ coordinates are 0 and $x$ is the eigenvector corresponding to the Laplacian eigenvalue 1 of $C_{n-t,r}$.
    By Proposition~\ref{prop:r=01-lollipop} (b),
    the coordinate $z_{t+ r'}$ of the vector $z$ is $-1$. 
    Since $z_t=0$ and  $z_{t+ r'}=-1$,
    we obtain $Az \neq 0$. This is a contradiction. Thus, $\mu_k(G) <1$ and hence $m_G[0,1) \geq k$.
\end{proof}
%%%%%%%%%%%%%%%%%%%%%%%%%%%%%%%%%%%%%%%%
\section{Conclusion}
%%%%%%%%%%%%%%%%%%%%%%%%%%%%%%%%%%%%%%%%%%
Let $G$ be a unicyclic graph with diameter $d(G)$ and girth $r$. 
Let $G'$ be a minimal unicyclic graph of $G$ that contains the diametral path. Then $G'$ is one of the cycle, lollipop, or compass graphs with diameter $d(G')=d(G)$ and girth $r$.
By combining results in Section~\ref{sec:main}, we have
\begin{equation*}
        m_{G'}[0,1) \geq \bigg\lceil \frac{d(G')}{3} \bigg\rceil+ \bigg\lceil \frac{r}{6} \bigg\rceil-1.
\end{equation*}
By Lemma~\ref{lem:rem-pendant}, we obtain
    \begin{equation*}
        m_{G}[0,1) \geq \bigg\lceil \frac{d(G)}{3} \bigg\rceil+ \bigg\lceil \frac{r}{6} \bigg\rceil-1.
    \end{equation*}
Thus we have the following theorem.
\begin{theorem}
    Let $G$ be a unicyclic graph with diameter $d(G)$ and girth $r$. Then 
    \begin{equation*}
        m_G[0,1)\geq \bigg\lceil \frac{d(G)}{3} \bigg\rceil + \bigg\lceil \frac{r}{6} \bigg\rceil-1.
    \end{equation*}
\end{theorem}

%=========================================================
% References
%=========================================================
\bibliographystyle{plain}

\end{document}